\documentclass[sn-mathphys.bst,Numbered]{sn-jnl}% Math and Physical Sciences Reference Style
\usepackage{graphicx}%
\usepackage{multirow}%
\usepackage{amsmath,amssymb,amsfonts}%
\usepackage{amsthm}%
\usepackage{mathrsfs}%
\usepackage[title]{appendix}%
\usepackage{xcolor}%
\usepackage{textcomp}%
\usepackage{manyfoot}%
\usepackage{booktabs}%
\usepackage{algorithm}%
\usepackage{algorithmicx}%
\usepackage{algpseudocode}%
\usepackage{listings}%
\DeclareMathOperator{\dist}{dist}

\theoremstyle{plain} 
\newtheorem{thm}{Theorem}[section] % The section option numbers the theorems by section.  Without it the theorems are numbered 1,2,3,...
\newtheorem{lem}{Lemma}[section]
\newtheorem{prop}{Proposition}[section]
\newtheorem{cor}{Corollary}[section]

\theoremstyle{definition}
\newtheorem{dfn}[thm]{Definition} % The thm command is optional, but without it, the definitions will be numbered independently of the theorems.

\newtheorem{rmk}[thm]{Remark}

\theoremstyle{thmstyleone}%
\theoremstyle{thmstyletwo}%

\theoremstyle{thmstylethree}%

\begin{document}

\title[Extremal Growth of Multiple Toeplitz Operators]{Extremal Growth of Multiple Toeplitz Operators and Applications to Numerical Stability of Approximation Schemes}

%%=============================================================%%
%% Prefix	-> \pfx{Dr}
%% GivenName	-> \fnm{Joergen W.}
%% Particle	-> \spfx{van der} -> surname prefix
%% FamilyName	-> \sur{Ploeg}
%% Suffix	-> \sfx{IV}
%% NatureName	-> \tanm{Poet Laureate} -> Title after name
%% Degrees	-> \dgr{MSc, PhD}
%% \author*[1,2]{\pfx{Dr} \fnm{Joergen W.} \spfx{van der} \sur{Ploeg} \sfx{IV} \tanm{Poet Laureate} 
%%                 \dgr{MSc, PhD}}\email{iauthor@gmail.com}
%%=============================================================%%

\author*[1]{\fnm{Yash} \sur{Rastogi}}\email{yrastogi@uchicago.edu}

\affil*[1]{\orgdiv{Department of Mathematics}, \orgname{University of Chicago}, \orgaddress{\street{5734 S. University Avenue}, \city{Chicago}, \postcode{60637}, \state{IL}, \country{USA}}}

%%==================================%%
%% sample for unstructured abstract %%
%%==================================%%

\abstract{The conversion of resolvent conditions into semigroup estimates is crucial in the stability analysis of hyperbolic partial differential equations. For two families of multiple Toeplitz operators, we relate the power bound with a resolvent condition of Kreiss-Ritt type. Furthermore, we show that the power bound is bounded above by a polynomial of the resolvent condition. The operators under investigation do not fall into a well-understood class, so our analysis utilizes explicit reproducing kernel techniques. Our methods apply \textit{mutatis mutandis} to composites of Toeplitz operators with polynomial symbol, which arise frequently in the numerical solution of initial value problems encountered in science and engineering.}


%%\pacs[JEL Classification]{D8, H51}


\maketitle

\section*{Statements and Declarations}
\begin{itemize}
\item Funding: This work was supported by NSF grant DMS-1851843.
\item Competing interests: The author has no financial or proprietary interests in any material discussed in this article.
\end{itemize}

\section{Introduction}
A key problem in the stability theory of numerical methods for the solution of hyperbolic partial differential equations is deriving semigroup estimates from resolvent-type estimates \cite{LN91}. The stability theory, due to Gustafsson, Kreiss, and Sundstrom (GKS), relies on the Laplace transform with respect to the time variable and consequently the associated stability estimates are limited to no initial data \cite{GKS72}. Therefore, a major thrust of research is to incorporate nonzero initial data into these GKS estimates, which are resolvent-type estimates, to produce semigroup estimates \cite{Tre84}. A fundamental tool in this analysis, and therefore well-posedness for Cauchy problems in the theory of partial differential equations, is the Kreiss Matrix Theorem, which provides the equivalence of uniform boundedness of semigroups and a particular resolvent estimate for a family of matrices. This issue is subtle as the GKS stability estimate is equivalent to a \textit{slightly stronger} version of the Kreiss resolvent condition, which is the following:
\begin{equation}\label{one}
    \sup_{z \in \mathbf{C}, |z|>1} (|z|-1) \|(zI-A)^{-1}\|_{\mathcal{L}(\ell^2(\mathbf{N}))} < + \infty
\end{equation}
where $T$ is some bounded operator on $\ell^2(\mathbf{N})$ that incorporates both the discretization of the hyperbolic
equation and the numerical boundary conditions \cite{Cou15}. In finite dimensions, the Kreiss Matrix Theorem asserts the equivalence of equation \ref{one} and power-boundedness of $A$, which is equivalent to deriving an optimal semigroup estimate. However, for bounded linear operators on a Banach space, condition \ref{one} implies only $\|A^n\| \in O(n)$ and therefore is insufficient to guarantee power-boundedness. In fact, the Hille-Yosida resolvent condition, stated below, is clearly stronger than that in \ref{one}, but is still insufficient to guarantee power-boundedness \cite{Cou15}:
\begin{equation}\label{two}
    \sup_{n \geq 1} \sup_{z \in \mathbf{C}, |z|>1} (|z|-1)^n \|(zI - A)^{-n}\|_{\mathcal{L}(\ell^2(\mathbf{N}))} < + \infty 
\end{equation}
An infinite-dimensional analogue of the Kreiss Matrix Theorem, proven by El-Fallah and Ransford as Corollary 1.3 of \cite{ER02} is the following:
\begin{thm}
Let $X$ be a complex Banach space and $A$ be a bounded linear operator on $X$. Suppose that $A$ satisfies a resolvent condition of the form 
\[\|(\lambda I - A)^{-1}\|\leq \frac{C}{\dist(\lambda, E)}\]
for some compact subset $E$ of the unit circle and some positive scalar $C$. If $E$ is finite, then 
\[\sup_{n \geq 0} \|A^n\| \leq \frac{e}{2}C^2 \# E\]
\end{thm}

In this paper, we enhance El-Fallah and Ransford's general relationship for infinite-dimensional Banach spaces, specifically addressing two families of multiple Toeplitz operators. Proving the stability of numerical approximation methods requires constructing families of operators that satisfy a particular resolvent condition and have uniformly bounded powers \cite{AN20, BS00, SW88}. Despite the fact that the power bound is not controlled by the resolvent condition, we demonstrate that the power bound for these operators has an upper bound determined by a polynomial in the resolvent condition. Our methods extend \textit{mutatis mutandis} to operators of the form $T_{g(z)}^{-1}T_{f(z)}T_{g(z)}$ where $f(z)$ is a polynomial in $z$ and $\overline{z}$ and $g(z)$ is a polynomial in $z$. This work is of interest since such operators arise in the numerical solution of initial value problems commonly encountered in science and engineering. 

We investigate families of operators that do not fall into a well-understood class. In particular, they are not normal, not Toeplitz, and not contractions. Consequently, standard results on special types of operators are inapplicable. We utilize the reproducing kernel Hilbert space structure of the Hardy space on the unit disk as well as properties of Toeplitz operators and operator norms to compute explicit estimates that would not be possible through more general functional analytic techniques. We demonstrate how power-bounded operators can be sourced from reproducing kernel spaces, characterize the extremal growth of their power bound and resolvent condition, and indicate the applicability of these results to numerical approximation methods. 

The structure of this paper is as follows. In Section 2, we provide background and a precise formulation of the research problem. In Section 3, we state the main results of the paper. Sections 4, 5, and 6 are devoted to the proofs of Theorems 3.1, 3.2, and 3.3, respectively. In Section 7, we detail the application of this work to the stability analysis of numerical approximation methods for initial value problems. 

\section{Background}
\subsection{Power-Boundedness and the Resolvent Condition}
    We define the power bound $M(A)$ and a resolvent condition $P(A)$ below. In this paper, we investigate the relation between $M(A)$ and $P(A)$ for two classes of operators.
    
    \begin{dfn}
        Let $A$ be a bounded operator with spectrum $\sigma(A)$ contained within the closed unit disk $\overline{\mathbf{D}}$ of $\mathbf{C}$. 
\begin{enumerate}[a)] 
  \item Set \[M(A) :=\sup_{n \geq 0} \|A^n\|\] If $M(A)<+\infty$, then $A$ is \textit{power-bounded}.
  \item We set 
  \[P(A):=\sup_{|\lambda|>1} \dist(\lambda, \sigma(A))\|(\lambda I - A)^{-1}\|\]
\end{enumerate}
    \end{dfn} 
All operators $A$ considered in the remainder of this paper will satisfy the condition $\sigma(A) \subseteq \overline{\mathbf{D}}$.
    \begin{rmk}
Diagonal operators illustrate the motivation for the expression $P(A)$. Let $D = \text{diag}(d_1, ..., d_n)$ with $\sigma(D) \subseteq \overline{\mathbf{D}}$. 
Then,
\[\|(zI- D)^{-1}\| = \max_j |(z-d_j)^{-1}| = \frac{1}{\min_j |z-d_j|} = \frac{1}{\dist(z, \sigma(D))}\]
This shows that $\|(zI-A)^{-1}\|$ and $\dist(z, \sigma(A))$ counteract each other and that they do so equally for diagonal operators, as $P(D)=1$. The association between these quantities can be seen in several papers such as \cite{DS06}, \cite{San96}, \cite{SS96}, and \cite{Zar09}.
\end{rmk}  

\begin{rmk}
For all operators $A$, note that $A^0$ is the identity operator, so $M(A) \geq \|A^0\|=1$. Furthermore, all operators $A$ also satisfy $P(A) \geq 1$. We prove this below in the manner of Corollary VII.3.3 of Dunford and Schwartz \cite{DS88}.  
\end{rmk}
\begin{prop}
\[1 \leq P(A)\]
\end{prop}
\begin{proof}
For $|\lambda|>1$, we have that 
\begin{align*} 
P(A) &\geq \dist(\lambda, \sigma(A))\|(\lambda I - A)^{-1}\| \\
&\geq  \dist(\lambda, \sigma(A)) \rho((\lambda I - A)^{-1}) \\
&= \dist(\lambda, \sigma(A)) \cdot \sup_{\varphi \in \mathfrak{M}} \left|\frac{1}{z-\varphi(A)}\right| \\
&= \dist(\lambda, \sigma(A))\cdot \frac{1}{\dist(\lambda, \sigma(A))} \\
&=1
\end{align*}
\end{proof}

\subsection{Hardy Space Toeplitz Operators}
\begin{dfn}
Let $\mathbf{D} \subseteq \mathbf{C}$ denote the open unit disk and let $\,d\theta$ denote the normalized arc length measure on the unit circle $\partial \mathbf{D}$. The \textit{Hardy space} $H^2$ is the closed linear span in $L^2(\partial \mathbf{D}, \theta)$ of $\{z^n: n \in \mathbf{N}\}$ where $z$ is the coordinate function on $\partial \mathbf{D}$, meaning $z(e^{it})=e^{it}$, and $\mathbf{N}$ is the set of nonnegative integers. For $f \in L^\infty(\partial \mathbf{D}, \theta)$, the \textit{Toeplitz operator with symbol $f$}, denoted $T_f$, is the operator on $H^2$ defined by 
\[T_fh = P(fh)\] 
where $P$ denotes the orthogonal projection of $L^2(\partial \mathbf{D}, \theta)$ onto $H^2$.  
\end{dfn}

The set $\{z^n: n \in \mathbf{N}\}$ forms an orthonormal basis for $H^2$ and the matrix of a Toeplitz operator with respect to it. If $f$ is a function in $L^\infty(\partial \mathbf{D})$ with Fourier coefficients $\hat{f}(n) = \int_{\partial \mathbf{D}} f \chi_{-n} \,d\theta$, then the matrix $\{a_{m,n}\}_{m,n \in \mathbf{N}}$ for $T_f$ with respect to the basis $\{\chi_n: n \in \mathbf{N}\}$ is 
\[a_{m,n} = (T_f \chi_n, \chi_m) = \int_{\partial \mathbf{D}} f \chi_{n-m} \,d\theta = \hat{f}(m-n)\]
The matrix coefficients satisfy $a_{m,n}=a_{m-n,0}$. Such a matrix is called a \textit{Toeplitz matrix}.

\subsection{Reproducing Kernel Hilbert Space}
The exposition in this section is based on \cite{Paulsen}. 
\begin{dfn}
Given a set $X$, we call $\mathcal{H}$ a \textit{reproducing kernel Hilbert space} (RKHS) over $\mathbf{C}$ provided the following hold. 
\begin{enumerate}[a)]
  \item $\mathcal{H}$ is a vector subspace of the space of linear functionals on $X$.
  \item $\mathcal{H}$ is endowed with an inner product, making it into a Hilbert space.
  \item For every $y \in X$, the linear evaluation functional $E_y: \mathcal{H} \to \mathbf{C}$ defined by $E_y(f) = f(y)$ is bounded. The vector that corresponds to this evaluation map is called the \textit{kernel vector corresponding to $g$}. An inner product with a kernel vector provides point evaluation in a RKHS. 
\end{enumerate}
\end{dfn}

The Hardy space $H^2$ is unitarily equivalent to the RKHS $H^2(\mathbf{D})$ over the open disk $\mathbf{D}$ in such a way that coordinate multiplication is preserved. Because of this equivalence, we identify the spaces $H^2(\mathbf{D})$ and $H^2$ and carry over the kernel vectors of $H^2(\mathbf{D})$ to $H^2$.  The kernel vector for $H^2(\mathbf{D})$ for a point $\omega \in \mathbf{D}$ is
\[k_{\omega}(z) = \sum_{n=0}^\infty \overline{\omega^n} z^n = \frac{1}{1-\overline{\omega}z}\]
because for any $f(z) = \sum_{n=0}^\infty a_n z^n \in H^2(\mathbf{D})$, the statement $\langle f, k_\omega \rangle = f(\omega)$ holds. For Toeplitz operators, we have the following property.
\begin{prop}
If $q$ is analytic and bounded on $\mathbf{D}$, then $T_q^\ast k_\omega= \overline{q(\omega)} k_\omega$ where $k_\omega$ is the kernel vector corresponding to $\omega \in \mathbf{D}$.
\end{prop}

\subsection{Research Problem}
Although it is known that $M(A)$ is not controlled by $P(A)$, a relationship of the form $M(A) \leq CP(A)^k$ where $C$ and $k$ are numerical constants holds for some classes of operators $A$. In this paper, we provide explicit bounds that relate $M(A)$ and $P(A)$ for the following families of operators $A$: 
\begin{itemize}
    \item $T_{1+\beta z}^{-1} T_{\overline{z}}T_{1+\beta z}$
    \item $T_{1+\beta z}^{-1} T_{\frac{z+\overline{z}}{2}}T_{1+\beta z}$
\end{itemize}
The choice of $1+\beta z$ is due to the following. 

\begin{rmk}
By linearity of Toeplitz operators, 
\[T_{a+bz}^{-1}T_fT_{a+bz} =T_{1+\frac{b}{a}z}^{-1}T_fT_{1+\frac{b}{a}z} \]
The two parameters $a$ and $b$ can be encapsulated into a single parameter, call it $\beta:=\frac{b}{a}$. Thus, we use $1+\beta z$ in the remainder of this paper for simplicity of notation. The operators we consider have spectrum in the closed unit disk, so $|\beta|<1$. Furthermore, the case $\beta=0$ is uninteresting since its undesirable consequence is that $A$ is a Toeplitz operator. Hence, we study our two families of operators when the parameter $\beta$ is such that $0 < |\beta|<1$.  
\end{rmk}

These two families of operators were selected since they contain infinitely many points in their spectra and do not fall into a well-understood class of operators. More precisely, our operators take the form $A = T_g^{-1}T_fT_g$, where $f$ and $g$ are selected such that $A$ is not a contraction, not normal, and not Toeplitz. In particular, if $f$ extends to be analytic on $\mathbf{D}$, then $A$ equals $T_f$. Thus, we are interested in the case where $f$ does not extend analytically to $\mathbf{D}$. For certain operators of the form of $A$, we have the following result. 
\begin{prop}
Let $A = T_g^{-1}T_fT_g$, $\|f\|_\infty \leq 1$, and $g$ analytic in $\overline{\mathbf{D}}$ such that $\sigma(A) = \overline{\mathbf{D}}$. Then, $P(A) \leq M(A)$.
\end{prop}
\begin{proof}
Observe that 
\[ \|(\lambda I - A)^{-1}\|  = \left\|\sum_{j=0}^\infty \frac{A^j}{\lambda^{j+1}}\right\| \leq \sum_{j=0}^\infty \frac{\|A^j\|}{|\lambda|^{j+1}} \leq \sum_{j=0}^\infty \frac{M(A)}{|\lambda|^{j+1}} =\frac{M(A)}{|\lambda|-1} \]

Note that $\sigma(A) = \overline{\mathbf{D}}$, so $\dist(\lambda, \sigma(A))=|\lambda|-1$. Hence, 
\[\dist(\lambda, \sigma(A))\|(\lambda I - A)^{-1}\| \leq M(A)\]
for all $|\lambda|>1$. In particular,  
\[P(A) = \sup_{|\lambda|>1} \dist(\lambda, \sigma(A))\|(\lambda I - A)^{-1}\| \leq M(A)\]
\end{proof}

\subsection{Notational Conventions}
In this section, we define for the reader some of the conventions employed in the remainder of the paper. 
\begin{dfn}
The \textit{commutator} of two operators $X$ and $Y$ is defined by $[X,Y]= XY-YX$.  
\end{dfn}
The commutator is an indicator of how significantly the operators $X$ and $Y$ fail to commute. It equals zero if and only if $X$ and $Y$ do commute.
\begin{dfn}
Given $u$ and $v$ in a Hilbert space $\mathcal{H}$, define the operator $uv^\ast$ on $\mathcal{H}$ by $uv^\ast h = \langle h,v \rangle u$. 
\end{dfn}

\section{Main Results}
The main results of this paper are the following three theorems. 
\begin{thm}
Let $A=T_{1+\beta z}^{-1}T_{\overline z}T_{1+\beta z}$. Then,
\[\max \left\{1, \frac{|\beta|}{\sqrt{1-|\beta|^2}} \right\}\leq M(A) \leq  1+\frac{|\beta|}{\sqrt{1-|\beta|^2}} \]
and 
\[\max \left\{1,\frac{c_0|\beta|}{\sqrt{1-|\beta|^2}}\right\} \leq P(A) \leq \min \left\{M(A),1+\frac{c_0|\beta|}{\sqrt{1-|\beta|^2}}\right\}\]
where $c_0 = \frac{\sqrt{2}(\sqrt{5}-1)}{(1+\sqrt{5})^{\frac{3}{2}}}$.
\end{thm}
\begin{thm}
Set \[P_S(A) := \sup_{|\lambda|>1} \dist(\lambda, S)\|(\lambda I - A)^{-1}\|\] where $S\subseteq \sigma(A) = [-1,1]$. Then, 
\begin{itemize}
    \item $P_{\sigma(A)}(A) \leq P_{\{-1,1\}}(A)$
    \item $M(A)\leq eP_{\{-1,1\}}(A)^2 \leq 2eP_{\sigma(A)}(A)^2$
\end{itemize}
\end{thm}
\begin{thm}
Let $A = T_{1+\beta z}^{-1}T_{\frac{z+\overline{z}}{2}}T_{1+\beta z}$. Then, 
\[
\max \left\{\frac{\left|\beta - \overline{\beta}(1-|\beta|^2)\right|}{2\sqrt{1-|\beta|^2}} , \frac{\sqrt{1-|\beta|^2}}{8} \left|\frac{2+\overline{\beta}^2}{1-|\beta|^2}-2\overline{\beta} - \overline{\beta}^3\right|\right\}\leq M(A)  \leq \frac{1+|\beta|}{1-|\beta|}\]
and
\[\max \left\{1, \sqrt{\frac{M(A)}{2e}} \right\} \leq P(A) \leq \frac{1+|\beta|}{1-|\beta|}.\]
\end{thm}
The following are corollaries of Theorems 3.1 and 3.3, respectively. They show that the two families of operators we study in this paper exhibit a relationship of the form $M(A) \leq C\cdot P(A)^k$ for numerical constants $C$ and $k$. 

\begin{cor}
    Set $A_\beta = T_{1+\beta z}^{-1} T_{\overline{z}} T_{1+\beta z}$. As $|\beta| \to 1^{-}$, $M(A_\beta)$ and $P(A_\beta)$ both grow in proportion to $\frac{1}{\sqrt{1-|\beta|}}$.
\end{cor}
\begin{cor}
Set $A_\beta = T_{1+\beta z}^{-1} T_{\frac{z+\overline{z}}{2}} T_{1+\beta z}$. As $|\beta| \to 1^{-}$, $M(A_\beta)$ grows in proportion of $\frac{1}{(1-|\beta|)^j}$ for $j \in [\frac{1}{2},1]$, so $P(A_\beta)$ grows in proportion to $\frac{1}{(1-|\beta|)^\ell}$ for $\ell \in [\frac{1}{4}, 1]$.
\end{cor}

\section{Proof of Theorem 3.1}
\subsection{Upper Bounds for the Power Bound}
\begin{prop}
If $A = S^{-1}BS$ where $B$ is a contraction, then $M(A)\leq \|S^{-1}\|\cdot \|S\|$. 
\end{prop}
\begin{proof}
By submultiplicativity of the operator norm, 
\[M(A) \leq \sup_{n \geq 0} \|S^{-1}\| \cdot \|B^n\| \cdot \|S\| \leq \sup_{n \geq 0} \|S^{-1}\| \cdot \|B\|^n \cdot \|S\| \leq \|S^{-1}\|\cdot \|S\| \]
\end{proof}
\begin{cor}
For $A = T_{1+\beta z}^{-1}T_{\overline{z}}T_{1+\beta z}$, 
\[M(A)\leq \frac{1+|\beta|}{1-|\beta|}\]
\end{cor}
\begin{proof}
Because $\|T_{1+\beta z}^{-1}\| = \frac{1}{1-|\beta|}$ and $\|T_{1+\beta z}\| = 1+|\beta|$, we have that $M(A)\leq \frac{1+|\beta|}{1-|\beta|}$.
\end{proof}

\begin{prop}
\[ M(A) \leq 1+\frac{|\beta|}{\sqrt{1-|\beta|^2}}\]
\end{prop}
\begin{proof}
Observe that 
\begin{align*} 
[(T_z^\ast)^n, T_z]  &= T_{\overline{z}}^nT_z - T_z T_{\overline{z}}^n \\ 
 &= T_{\overline{z}}^{n-1} - (I - e_0 e_0^\ast)T_{\overline{z}}^{n-1} \\
 &= e_0 e_0^\ast (T_z^{n-1})^\ast \\
 &= e_0 (T_z^{n-1}e_0)^\ast \\
 &= e_0 e_{n-1}^\ast
\end{align*}
 Therefore,
\begin{align*} 
A^n &= T_{1+\beta z}^{-1}T_f^n T_{1+\beta z} \\
&= T_f^n + T_{1+\beta z}^{-1}(T_f^nT_{1+\beta z} - T_{1+\beta z}T_f^n) \\
&= T_f^n + T_{1+\beta z}^{-1}[T_{\overline{z}}, T_{1+\beta z}] \\
&=T_f^n + T_{1+\beta z}^{-1}([T_{\overline{z}}, T_1]+\beta[T_{\overline{z}}, T_z])\\
&= (T_z^\ast)^n  + \beta T_{1+\beta z}^{-1} [T_{\overline{z}}, T_z] \\
&= (T_z^\ast)^n + \beta T_{1+\beta z}^{-1}e_0 (e_{n-1}^\ast)  \\ 
 &=  (T_z^\ast)^n + \frac{\beta}{1+\beta z} (e_{n-1}^\ast)  \\
 &= (T_z^\ast)^n +\beta k_{-\overline{\beta}}(e_{n-1}^\ast)
\end{align*}

By the Triangle Inequality and submultiplicativity,
\begin{align*} 
\|A^n\| &\leq   \|(T_z^\ast)^n\| + |\beta|\cdot  \|k_{-\overline{\beta}}(e_{n-1}^\ast)\| \\ 
 &\leq 1+ |\beta|\cdot  \|k_{-\overline{\beta}}\|\cdot \|(e_{n-1}^\ast)\|  \\
 &= 1+|\beta|\cdot \|k_{-\overline{\beta}}\|  \\
 &=  1+\frac{|\beta|}{\sqrt{1-|\beta|^2}}
\end{align*}
Substituting for $\beta$ gives the result. 
\end{proof}

\subsection{Lower Bounds for the Power Bound}

\begin{prop}
\[\frac{|\beta|}{\sqrt{1-|\beta|^2}} \leq M(A)\]
\end{prop}
\begin{proof}
We consider $\langle T_{\frac{1}{1+\beta z}}T_{\overline{z}}^n T_{1+\beta z} z^{n-1}, \widehat{k_\omega} \rangle $ where $\omega \in \mathbf{D}$ and $\widehat{k_\omega} = \frac{k_\omega}{\|k_\omega\|}$ because 
\[\left|\langle T_{\frac{1}{1+\beta z}}T_{\overline{z}}^n T_{1+\beta z} z^{n-1}, \widehat{k_\omega} \rangle\right| = \left|\frac{\langle T_{\frac{1}{1+\beta z}}T_{\overline{z}}^n T_{1+\beta z} z^{n-1}, \widehat{k_\omega} \rangle}{\|z^{n-1}\|\cdot \|\widehat{k_\omega}\|}\right| \leq \|A^n\| \]
Observe that 
\begin{align*} 
\langle T_{\frac{1}{1+\beta z}}T_{\overline{z}}^n T_{1+\beta z} z^{n-1}, \widehat{k_\omega} \rangle  &= \langle T_{\overline{z}}^n T_{1+\beta z} z^{n-1}, T_{\frac{1}{1+\beta z}}^\ast \widehat{k_\omega} \rangle \\ 
 &= \langle T_{\overline{z}}^n T_{1+\beta z} z^{n-1}, \frac{1}{\overline{1+\beta \omega}} \widehat{k_\omega} \rangle \\
 &= \frac{1}{1+\beta \omega} \langle T_{\overline{z}}^n T_{1+\beta z} z^{n-1}, \widehat{k_\omega} \rangle \\
 &= \frac{1}{1+\beta \omega} \langle \beta, \widehat{k_\omega} \rangle 
\end{align*}
Therefore, for any $\omega \in \mathbf{D}$,
\[|\langle T_{\frac{1}{1+\beta z}}T_{\overline{z}}^n T_{1+\beta z} z^{n-1}, \widehat{k_{\omega}} \rangle| =  \frac{|\beta|\sqrt{1-|\omega|^2}}{|1+\beta \omega|}\]
Now, for $\omega = re^{i\theta}$ we define $\text{sgn}(\omega) = e^{i\theta}$. Pick $\text{sgn} (\omega) = -\text{sgn}\left(\frac{1}{\beta}\right)$ to obtain  
\[|\langle T_{\frac{1}{1+\beta z}}T_{\overline{z}}^n T_{1+\beta z} z^{n-1}, \widehat{k_{\omega}} \rangle| =  \frac{\sqrt{1-|\omega|^2}}{\left|\frac{1}{\beta}\right|-|\omega|}\]
Using calculus of a real variable, the above expression is maximized at $\omega = |\beta|$. Substituting this into the above gives $\frac{|\beta|}{\sqrt{1-|\beta|^2}}$. Therefore, 
\[M(A)\geq \frac{|\beta|}{\sqrt{1-|\beta|^2}}\]
\end{proof}
\begin{prop}
\[1 \leq M(A) \]
\end{prop}
\begin{proof}
This follows from $A$ being a contraction. 
\end{proof}
\subsection{Upper Bounds for the Resolvent Condition}
\begin{prop}
\[P(A)\leq M(A)\]
\end{prop}
\begin{proof}
See proposition 2.2.
\end{proof}

\begin{lem}
For bounded operators $B$ and $S$, \[(\lambda I - S^{-1}BS)^{-1} = (\lambda I - B)^{-1} + S^{-1}[(\lambda I - S)^{-1}, S]\]
\end{lem}
\begin{proof}
Observe that 
\begin{align*} 
(\lambda I - S^{-1}BS)^{-1} &= (\lambda S^{-1}S - S^{-1}BS)^{-1} \\ 
 &=  [S^{-1}(\lambda I - B)S]^{-1} \\
 &= S^{-1}(\lambda I - B)^{-1}S \\
 &= (\lambda I - B)^{-1} + S^{-1} [(\lambda I - B)^{-1}, S]
\end{align*}
\end{proof}

\begin{lem}
$(\lambda I - A)^{-1} = T_{\frac{1}{\lambda - \overline{z}}} + \frac{\beta}{\lambda^2}k_{-\overline{\beta}} k_{\frac{1}{\lambda}}^\ast$
\end{lem}
\begin{proof}
By lemma 4.1, 
\[(\lambda I - A)^{-1} = T_{\frac{1}{\lambda - \overline{z}}} + T_{1+\beta z}^{-1}[T_{\frac{1}{\lambda - \overline{z}}}, T_{1+\beta z}]\]
Observe that 
\begin{align*} 
[T_{\frac{1}{\lambda - \overline{z}}}, T_{1+\beta z}] &=\beta [T_{\frac{1}{\lambda - \overline{z}}}, T_z]  \\ 
 &=  \beta \sum_{k=0}^\infty \frac{1}{\lambda^{k+1}} [(T_z^\ast)^k, T_z] \\
 &= \beta \sum_{k=1}^\infty \frac{1}{\lambda^{k+1}} e_0 (e_{k-1}^\ast) \\
 &= \beta e_0 \left(\sum_{k=1}^\infty \frac{1}{\overline{\lambda}^{k+1}} e_{k-1}\right)^\ast \\
 &= \beta e_0 \left(\frac{1}{\overline{\lambda}^2} \sum_{k=1}^\infty \frac{1}{\overline{\lambda}^{k-1}} e_{k-1}\right)^\ast \\
 &= \beta e_0 \left(\frac{1}{\overline{\lambda}^2}\cdot  \frac{1}{1-\frac{z}{\overline{\lambda}}}\right)^\ast 
\end{align*}
Substituting, we get that 
\begin{align*} 
(\lambda I - A)^{-1} &= T_{\frac{1}{\lambda - \overline{z}}} + T_{1+\beta z}^{-1}\beta e_0 \left(\frac{1}{\overline{\lambda}^2}\cdot  \frac{1}{1-\frac{z}{\overline{\lambda}}}\right)^\ast  \\ 
 &=  T_{\frac{1}{\lambda - \overline{z}}} + \frac{\beta}{\lambda^2} \left(\frac{1}{1+\beta z}\right)k_{\frac{1}{\lambda}}^\ast \\
 &= T_{\frac{1}{\lambda - \overline{z}}} + \frac{\beta}{\lambda^2} k_{-\overline{\beta}} k_{\frac{1}{\lambda}}^\ast 
\end{align*}
\end{proof}
\begin{prop}
\[P(A) \leq 1+\frac{c_0 |\beta|}{\sqrt{1-|\beta|^2}}\]
\end{prop}
\begin{proof}
Observe that 
\begin{align*} 
\|T_{1+\beta z}^{-1}T_{\frac{1}{\lambda - \overline{z}}}T_{1+\beta z}
\| &\leq \|T_{\frac{1}{\lambda - \overline{z}}}\| + \left\|\frac{\beta}{\lambda^2} k_{-\overline{\beta}}k_{\frac{1}{\lambda}}^\ast\right\|  \\ 
 &=  \frac{1}{|\lambda|-1} + \frac{|\beta|}{|\lambda|^2}\cdot \|k_{-\overline{\beta}}\| \cdot \|k_{\frac{1}{\lambda}}\| \\
 &= \frac{1}{|\lambda|-1} + \frac{|\beta|}{|\lambda|^2} \cdot \frac{1}{\sqrt{1-|\beta|^2}\sqrt{1-|\frac{1}{\lambda}|^2}} \\
 &= \frac{1}{|\lambda|-1} + \frac{|\beta|}{|\lambda| \sqrt{1-|\beta|^2}\sqrt{|\lambda|^2-1}}
\end{align*}
Therefore, 
\[(|\lambda|-1) \|T_{1+\beta z}^{-1}T_{\frac{1}{\lambda - \overline{z}}}T_{1+\beta z}\| \leq 1+\frac{|\beta|}{\sqrt{1-|\beta|^2}}\cdot \frac{|\lambda|-1}{|\lambda|\sqrt{|\lambda|^2-1}}\]
Since $\sigma(A) = \overline{\mathbf{D}}$, it follows that  
\[P(A) \leq \sup_{|\lambda|>1} 1+\frac{|\beta|}{\sqrt{1-|\beta|^2}}\cdot \frac{|\lambda|-1}{|\lambda|\sqrt{|\lambda|^2-1}}\]
Using single variable calculus, we maximize $\frac{|\lambda|-1}{|\lambda|\sqrt{|\lambda|^2-1}}$ as a function of $|\lambda|$. The optimal value occurs at $|\lambda|=\frac{1+\sqrt{5}}{2}$ and hence 
\[P(A)\leq 1+ \left(\frac{\sqrt{2}(\sqrt{5}-1)}{(1+\sqrt{5})^{\frac{3}{2}}}\right)\frac{|\beta|}{\sqrt{1-|\beta|^2}}\]
\end{proof}
\subsection{Lower Bounds for the Resolvent Condition}
\begin{prop}
\[P(A)\geq 1\]
\end{prop}
\begin{proof}
Because $A$ is a contraction, the result is immediate from Proposition 1.1.
\end{proof}
\begin{prop}
\[\frac{c_0|\beta|}{\sqrt{1-|\beta|^2}} \leq P(A)\]
\end{prop}
\begin{proof}
First, we estimate the norm of the resolvent of $A$. Recall from Lemma 4.2 that 
\[(\lambda I- A)^{-1} =  T_{\frac{1}{\lambda - \overline{z}}} + \frac{\beta}{\lambda^2} k_{-\overline{\beta}} k_{\frac{1}{\lambda}}^\ast \]
In this proof, we first estimate the norm of the resolvent and then use that to estimate $P(A)$. Let $Q(\lambda)$ be given by 
\[Q(\lambda) = \frac{\langle (T_{\frac{1}{\lambda -\overline{z}}} + \frac{\beta}{\lambda^2}k_{-\overline{\beta}}k_{\frac{1}{\lambda}}^\ast) \lambda^2 k_{\frac{1}{\lambda}}, \beta k_{-\overline{\beta}}\rangle}{\left(\frac{|\lambda|^2}{\sqrt{1-\frac{1}{|\lambda|^2}}}\right)\left(\frac{|\beta|}{\sqrt{1-|\beta|^2}}\right)}\]
We have 
\begin{align*} 
Q(\lambda) &= \frac{\sqrt{|\lambda|^2-1}\sqrt{1-|\beta|^2}}{|\lambda|^3|\beta|} \langle T_{\frac{1}{\lambda - \overline{z}}}k_{\frac{1}{\lambda}}, k_{-\overline{\beta}}\rangle \lambda^2 \overline{\beta} + \frac{\left(1-\frac{1}{|\lambda|^2}\right)^{-1}}{\frac{|\lambda|^2}{\left(1-\frac{1}{|\lambda|^2}\right)^{\frac{1}{2}}}}\frac{|\beta|^2\cdot \frac{1}{1-|\beta|^2}}{\frac{|\beta|}{\sqrt{1-|\beta|^2}}}\\ 
 &= \frac{\sqrt{|\lambda|^2-1}\sqrt{1-|\beta|^2}\lambda^2\overline{\beta}}{|\lambda|^3|\beta|}\frac{1}{\lambda - \frac{1}{\overline{\lambda}}}\frac{1}{1+\frac{\overline{\beta}}{\overline{\lambda}}} + \frac{1}{|\lambda|\sqrt{|\lambda|^2-1}}\frac{|\beta|}{\sqrt{1-|\beta|^2}} \\
 &= \frac{\sqrt{|\lambda|^2-1}\sqrt{1-|\beta|^2}}{|\lambda|^3|\beta|}\frac{\overline{\lambda}^2\lambda^2\overline{\beta}}{(|\lambda|^2-1)(\overline{\lambda + \beta})} + \frac{1}{|\lambda|\sqrt{|\lambda|^2-1}}\frac{|\beta|}{\sqrt{1-|\beta|^2}} \\
 &= \frac{\frac{\overline{\beta}}{|\beta|} |\lambda| \sqrt{1-|\beta|^2}}{\sqrt{|\lambda|^2-1}(\overline{\lambda}+\overline{\beta})} + \frac{1}{|\lambda|\sqrt{|\lambda|^2-1}} \frac{|\beta|}{\sqrt{1-|\beta|^2}} \\
 &= \frac{F_1(|\lambda|)}{\overline{\lambda}+\overline{\beta}} + F_2(|\lambda|)
\end{align*}
where we have defined 
\[F_1(|\lambda|) = \frac{\frac{\overline{\beta}}{|\beta|} |\lambda| \sqrt{1-|\beta|^2}}{\sqrt{|\lambda|^2-1}}\]
and 
\[F_2(|\lambda|)= \frac{1}{|\lambda|\sqrt{|\lambda|^2-1}}\cdot \frac{|\beta|}{\sqrt{1-|\beta|^2}}\]
By definition of $P(A)$, we know that
\[|Q(\lambda)| \leq \frac{P(A)}{|\lambda|-1}\]
So, 
\[|Q(\lambda)|^2 \leq \frac{P(A)^2}{(|\lambda|-1)^2}\]
Taking the average value yields
\[\frac{1}{2\pi} \int_0^{2\pi} |Q(re^{it})|^2 \,dt \leq \frac{P(A)^2}{(r-1)^2}\]
Observe that 
\begin{align*} 
\int_0^{2\pi} |Q(re^{it})|^2 \,dt &= \int_0^{2\pi} \left|\frac{F_1(r)}{re^{-it}+\overline{\beta}} + F_2(r)\right|^2 \,dt  \\ 
&= \int_0^{2\pi} \left| \frac{\frac{1}{r}F_1(r)e^{it}}{1+\left(\frac{\overline{\beta}}{r}\right)e^{it}} + F_2(r)\right|^2 \,dt
\end{align*}
Since $e^{it} \mapsto \frac{\frac{1}{r}F_1(r)e^{it}}{1+\left(\frac{\overline{\beta}}{r}\right)e^{it}}$ and $e^{it} \mapsto F_2(r)$ are orthogonal as elements of $H^2$, it follows by the Pythagorean Theorem that 
\[\int_0^{2\pi} |Q(re^{it})|^2 \,dt = \int_0^{2\pi} \left|\frac{\frac{1}{r}F_1(r)}{1 + \frac{\overline{\beta}}{r}e^{it}}\right|^2 + F_2(r)^2 \,dt \]
Since $\left|\frac{\frac{1}{r}F_1(r)}{1 + \frac{\overline{\beta}}{r}e^{it}}\right|^2\geq 0$, it follows that  
\[\int_0^{2\pi} \left|\frac{\frac{1}{r}F_1(r)}{1 + \frac{\overline{\beta}}{r}e^{it}}\right|^2 + F_2(r)^2 \,dt \geq 2\pi \cdot F_2(r)^2\]
Therefore, we have shown that 
\[F_2(r)^2 \leq \frac{1}{2\pi} \int_0^{2\pi} |Q(re^{it})|^2\,dt \leq \frac{P(A)^2}{(r-1)^2}\]
and hence 
\[(r-1)F_2(r) \leq P(A)\]
for all $r>1$. Through single-variable calculus, we find that the maximum of $(r-1)F_2(r)$ is equal to $c_0 = \frac{\sqrt{2}(\sqrt{5}-1)}{(1+\sqrt{5})^{\frac{3}{2}}}$ and occurs when $r=\frac{1+\sqrt{5}}{2}$. Hence, 
\[\frac{c_0|\beta|}{\sqrt{1-|\beta|^2}} \leq P(A)\]
\end{proof}

\section{Proof of Theorem 3.2}
Here, we prove results inspired by El-Fallah and Ransford \cite{ER02}. 
\begin{prop}
Let $A$ be a bounded operator with the property that $\sigma(A) = [-1,1]$. 
    \item $P_{\sigma(A)}(A) \leq P_{\{-1,1\}}(A)$
    \item $M(A)\leq eP_{\{-1,1\}}(A)^2 \leq 2eP_{\sigma(A)}(A)^2$
\end{prop}
\begin{proof}
By definition, $\dist(\lambda, \sigma(A)) = \dist(\lambda, [-1,1])$. Observe that 
\[ \dist(\lambda, A) = \begin{cases} 
     \sqrt{(1-|\text{Re}\lambda|)^2+(\text{Im}\lambda)^2}  &  |\text{Re}\lambda|< 1 \\
     \dist(\lambda, \{-1,1\})  &  |\text{Re}\lambda|\geq 1
   \end{cases}
\]
For $\lambda$ such that $|\text{Re}\lambda|<1$, we have $\dist(\lambda, \{-1,1\})= |\text{Im}\lambda|$. Obviously, $\sqrt{(1-|\text{Re}\lambda|)^2+(\text{Im}\lambda)^2} \geq \sqrt{(\text{Im}\lambda)^2} = |\text{Im}\lambda|$, so 
\[\dist(\lambda, \sigma(A)) \leq \dist(\lambda, \{-1,1\})\]
It follows that 
\begin{align*} 
P_{\sigma(A)}(A) &=  \sup_{|\lambda|>1} \dist(\lambda, \sigma(A))\|(\lambda I-A)^{-1} \| \\
 &\leq   \sup_{|\lambda|>1} \dist(\lambda, \{-1,1\})\|(\lambda I - A)^{-1}\| \\
 &=  P_{\{-1,1\}}(A)
\end{align*}
Now, we prove the first equality of the second result. Recall the result of El-Fallah and Ransford, which states that for a finite set $S$,
\[M(A)\leq \frac{e}{2}C \#S\] where $P_S(A) \leq C$. It follows that \[M(A) \leq \frac{e}{2}P_S(A)^2 \#S\]
Set $S = \{-1,1\}$ to get
\[M(A) \leq e P_{\{-1,1\}}(A)^2\]
Finally, we prove the second equality of the second result. Write $\lambda = x+iy$. The proof has two cases: $|x| \geq 1$ and $|x|<1$. If $|x| \geq 1$, then the point of the spectrum of $A$ that is closest to $\lambda$ is either of $\{-1,1\}$, so \[\dist(\lambda, \{-1,1\}) = \dist(\lambda, [-1,1]) = \sqrt{(|x|-1)^2+y^2}\]
Now, consider the case when $|x|<1$ and $|\lambda|>1$ and without loss of generality, assume that $x\geq 0$. All such $\lambda$ must satisfy 
$1-x \leq y$, hence $(1-x)^2 \leq y^2$, hence $(1-x)^2+y^2 \leq 2y^2$, and hence 
\[\dist(\lambda, \{-1,1\}) = \sqrt{(1-x)^2+y^2} \leq \sqrt{2}y =  \sqrt{2}\dist(\lambda, [-1,1])\]
Therefore, 
\[\dist(\lambda, \{-1,1\})^2 \leq 2\dist(\lambda, [-1,1])^2 \]
hence 
\[\sup_{|\lambda|>1} \dist(\lambda, \{-1,1\})^2 \|(\lambda I - A)^{-1}\|^2 \leq \sup_{|\lambda|>1} 2\dist(\lambda, [-1,1])^2 \|(\lambda I - A)^{-1}\|^2\]
which means that 
\[P_{\{-1,1\}}(A)^2 \leq 2 P_{\sigma(A)}^2\]
We have shown that 
\[M(A)\leq eP_{\{-1,1\}}(A)^2 \leq 2e P_{\sigma(A)}(A)^2\]
and since $P(A) = P_{\sigma(A)}(A)$, it follows immediately that 
\[\sqrt{\frac{M(A)}{2e}} \leq P(A)\]
\end{proof}
\section{Proof of Theorem 3.3}
\subsection{Estimates for the Power Bound}
\begin{prop}
\[M(A)\leq \frac{1+|\beta|}{1-|\beta|}\]
\end{prop}
\begin{proof}
By submultiplicativity of the operator norm, 
\[\|T_{\frac{z+\overline{z}}{2}}^n\|= \|T_{\frac{z+\overline{z}}{2}}\|^n \leq \left\|\frac{z+\overline{z}}{2}\right\|_\infty^n =1\]
Because $T_{\frac{z+\overline{z}}{2}}$ is a contraction, we can apply Proposition 4.1 and obtain 
\[M(A) \leq \|T_{1+\beta z}^{-1}\|\cdot \|T_{1+\beta z}\| = \frac{1+|\beta|}{1-|\beta|}\]
\end{proof}
We now reproduce the result of Proposition 6.1 more explicitly through considering the powers of the commutator. The following lemma will be utilized in the computation. 
\begin{lem}
The expansion for $[(T_z+T_z^*)^n, T_z]$ has $2(n+1)$ terms. 
\end{lem}
\begin{proof}
  Consider a polynomial $p$ with $n$ terms such that the product of any two of these terms equals the product of the same pair of terms in the reverse order. By the Stars and Bars Theorem of combinatorics, $p^k$ is expressible (through simplification) to a polynomial with $\binom{n+k-1}{k-1}$ terms. Since every word in $T_z$ and $T_z^\ast$ can be reduced to a product of the form $T_z^k\cdot T_{\overline z}^l$ or $T_z^l\cdot T_{\overline z}^k$, the expansion of $(T_z + T_z^\ast)^n$ has twice as many terms as it would have if $T_z$ and $T_z^\ast$ did commute. Hence, $(T_z + T_z^\ast)^n$ has $n+1$ terms, which implies $[(T_z+T_z^*)^n, T_z]$ has $2(n+1)$ terms.
\end{proof}

\begin{prop}
\[M(A)\leq \frac{1+|\beta|}{1-|\beta|}\]
\end{prop}
\begin{proof}
\begin{align*} 
M(A) &= \sup_{n \geq 0} \|T_{1+\beta z}^{-1}T_f^n T_{1+\beta z}\|  \\ 
 &=  \sup_{n \geq 0} \|T_{\frac{z+\overline{z}}{2}}^n + T_{1+\beta z}^{-1}[T_{\frac{z+\overline{z}}{2}}^n, T_{1+\beta z}]\| \\
 &\leq \sup_{n \geq 0} \frac{\|T_{z+\overline{z}}^n\| + |\beta|\cdot \|T_{1+\beta z}^{-1}\|\cdot \|[(T_z+T_{\overline{z}})^n, T_z]\|}{2^n} \\
 &= 1+ \frac{|\beta|}{1-|\beta|}\sup_{n \geq 0} \frac{\|[(T_z+T_{\overline{z}})^n, T_z]\|}{2^n} \\
 &\leq 1+\frac{|\beta|}{1-|\beta|} \sup_{n \geq 0}\frac{2(n+1)}{2^n} \\
 &= 1+\frac{2|\beta|}{1-|\beta|}
\end{align*}
The penultimate step follows from an application of Lemma 6.1 and the final step follows from single-variable calculus. 
\end{proof}
\begin{prop}
\[M(A) \geq \frac{\left|\beta - \overline{\beta} (1-|\beta|^2) \right|}{2\sqrt{1-|\beta|^2}}\]
\end{prop}
\begin{proof}
Observe that $A = T_{\frac{z+\overline{z}}{2}} + T_{1+\beta z}^{-1} [T_{\frac{z+\overline{z}}{2}}, T_{1+\beta z}]$. Clearly, \[[T_{\frac{z+\overline{z}}{2}}, T_{1+\beta z}] = \frac{\beta}{2}[T_{\overline{z}}, T_z] = \frac{\beta}{2}(I - T_zT_z^\ast)  = \frac{\beta}{2} e_0e_0^\ast \]
By substitution, 
\[A = T_{\frac{z+\overline{z}}{2}} + \frac{\beta}{2} k_{-\overline{\beta}}e_0^\ast\]
Hence, 
\begin{align*} 
\|A\| &\geq \frac{|\langle (T_{\frac{z+\overline{z}}{2}} + \frac{\beta}{2}k_{-\overline{\beta}}e_0^\ast) e_0,  k_{-\overline{\beta}} \rangle |}{\|k_{-\overline{\beta}}\|\cdot \|e_0\|} \\ 
 &=  \frac{|\langle T_{\frac{z}{2}}e_0 + \frac{\beta}{2} k_{-\overline{\beta}}\|e_0\|^2, k_{-\overline{\beta}}\rangle |}{\|k_{-\overline{\beta}}\|} \\
 &= \frac{|\langle T_{\frac{z}{2}}e_0 + \frac{\beta}{2} k_{-\overline{\beta}}, k_{-\overline{\beta}}\rangle |}{\|k_{-\overline{\beta}}\|} \\
 &= \frac{|\langle 1, T_{\frac{z}{2}}^\ast k_{-\overline{\beta}} \rangle + \frac{\beta}{2} \langle k_{-\overline{\beta}}, k_{-\overline{\beta}} \rangle|}{\|k_{-\overline{\beta}}\|} \\
 &= \frac{|\langle 1, \frac{-\beta}{2}k_{-\overline{\beta}} \rangle + \frac{\beta}{2} \| k_{-\overline{\beta}}\|^2 |}{\|k_{-\overline{\beta}}\|} \\
 &= \frac{|\frac{-\overline{\beta}}{2} + \frac{\beta}{2} \|k_{-\overline{\beta}}\|^2 |}{\|k_{-\overline{\beta}}\|} \\
 &= \frac{1}{2} \left|\beta \|k_{-\overline{\beta}}\| - \frac{\overline{\beta}}{\|k_{-\overline{\beta}}\|}\right| \\
 &= \frac{1}{2}\left|\beta \left(\frac{1}{\sqrt{1-|\beta|^2}}\right) - \overline{\beta} \sqrt{1-|\beta|^2}\right| \\
 &= \frac{|\beta - \overline{\beta}(1-|\beta|^2)|}{2\sqrt{1-|\beta|^2}}
\end{align*}
Since $\|A\| \leq M(A)$, 
\[M(A) \geq \frac{|\beta - \overline{\beta}(1-|\beta|^2)|}{2\sqrt{1-|\beta|^2}}\]
\end{proof}

\begin{prop}
\[M(A) \geq \frac{\sqrt{1-|\beta|^2}}{8}\cdot \left|\frac{2-\overline{\beta}^3}{1-|\beta|^2}-2\overline{\beta} - \overline{\beta}^3\right| \]
\end{prop}
\begin{proof}
Note that 
\[\|A^3\| \geq \frac{\sqrt{1-|\beta|^2}}{8}  \left| \Big\langle \left((T_z + T_{\overline{z}})^3 +  T_{1+\beta z}^{-1}[(T_z + T_{\overline{z}})^3]\right) 1, \widehat{k_{-\overline{\beta}}}\Big\rangle\right|\]
Observe that $[(T_z + T_{\overline{z}})^3, T_z] 1 =z^2+2$. Furthermore, expansion gives us $(T_z+T_{\overline{z}})^3 1 = z^3+2z$. So, 
\[\|A^3\| \geq \frac{\sqrt{1-|\beta|^2}}{8} \left|\Big\langle z^3 + 2z +\frac{z^2+2}{1+\beta z}, \widehat{k_{-\overline{\beta}}}\Big\rangle \right|\]
and hence
\[M(A) \geq \frac{\sqrt{1-|\beta|^2}}{8} \left|\frac{2+\overline{\beta}^2}{1-|\beta|^2}-2\overline{\beta} - \overline{\beta}^3\right| \]

\end{proof}

\subsection{Estimates for the Resolvent Condition}
\begin{prop}
\[1 \leq P(A)\]
\end{prop}
\begin{proof}
Since $A$ is a contraction, we apply Proposition 1.1 and obtain this result immediately. 
\end{proof}

\begin{prop}
\[\sqrt{\frac{M(A)}{2e}} \leq P(A)\]
\end{prop}
\begin{proof}
Since $f(z) = \frac{z+\overline{z}}{2}$ is real, $T_{f}$ is self-adjoint. This implies that $\sigma(T_f) \in \mathbf{R}$. The Hartman-Wintner Theorem, stated on page 163 of \cite{Dou98}, indicates that $\sigma(T_f) = [-1,1]$. Because $T_{1+\beta z}$ is an invertible operator, $\sigma(T_{1+\beta z}^{-1}T_fT_{1+\beta z})= [-1,1]$. Applying Proposition 5.1 completes the proof.
\end{proof}

\begin{prop}
\[P(A) \leq \frac{1+|\beta|}{1-|\beta|}\]
\end{prop}
\begin{proof}
\begin{align*} 
\|(\lambda I -A)^{-1}\| &= \|(T_{1+\beta z}^{-1}(\lambda I  - T_f) T_{1+\beta z})^{-1}\| \\
&= \|T_{1+\beta z}^{-1}(\lambda I - T_f)^{-1} T_{1+\beta z}\| \\
&\leq \|T_{1+\beta z}^{-1}\|\cdot \|T_{1+\beta z}\| \cdot \|(\lambda I - T_f)^{-1}\| \\
&= \frac{|1|+|\beta|}{|1|-|\beta|}\cdot \|(\lambda I - T_f)^{-1}\|
\end{align*}
It follows from Brown and Halmos' characterization of normal Toeplitz operators in \cite{Dou98} that $\lambda I - T_f$ is normal. The inverse of a normal operator is also normal, so $(\lambda I - T_f)^{-1}$ is normal. Now, observe that 
\[\|(\lambda I - T_f)^{-1}\| = \rho((\lambda I - T_f)^{-1}) = \frac{1}{\dist(\lambda,\sigma( T_f))} = \frac{1}{\dist(\lambda, \sigma(A))}\]
By substitution, 
\[P(A)\leq \sup_{|\lambda|>1} \frac{1+|\beta|}{1-|\beta|}(\dist(\lambda, \sigma(A)))\cdot \frac{1}{\dist(\lambda, \sigma(A))} = \frac{1+|\beta|}{1-|\beta|}\]
\end{proof}
\section{Stability Analysis of Methods for the Numerical Solution of the Cauchy Problem for Linear Differential Equations}
Several classes of nonlinear dynamical systems, such as time delays \cite{JOR19}, arise frequently in the natural sciences and engineering yet are difficult to model. Therefore, the stability analysis of numerical approximations for the differential equations used to model these systems is critical. Proving such stability results requires demonstrating the power-boundedness of an operator that results from numerical approximations \cite{BS00, DKS93, Spi97}. As a result, much recent work has gone into algorithmic approaches for accurate computation of the Kreiss constant (e.g. \cite{AN20}, \cite{Mit20}) and into sharpening the analytic conditions required for power-boundedness (e.g. \cite{BM21}, \cite{CCEL20}). Proofs of the numerical stability of approximation methods for particular examples can be found in \cite{DKS93} and \cite{Spi97} for linear initial value problems and \cite{RT90} for the method of lines approach, which is a spectral method (as opposed to a finite difference method) that reduces partial differential equations in space and time to a system of ordinary differential equations in time. 

Numerical methods for solving linear initial value problems produce equations that can typically be reduced to the following recursive numerical process:
\begin{equation}\label{three}
   u_n = Bu_{n-1} + b_n
\end{equation}
where the $b_n \in \mathbf{C}^s$, $B \in \mathcal{L}(\mathbf{C}^s)$, and the $u_n \in \mathbf{C}^s$ are computed recursively starting from an initial guess $u_0$. We are concerned with determining whether the \textit{propagated error} $v_n= \Tilde{u}_n - u_n$, resulting from a perturbed initial vector, is finite. Symbolically, suppose we start with $\Tilde{u_0}$ instead of $u_0$. By applying \ref{three}, we have   
\[v_n = \Tilde{u}_n - u_n  = (B\Tilde{u}_{n-1} + b_n ) - (Bu_{n-1} + b_n) = Bv_{n-1}\]
so $v_n = B^n v_0$. Thus, for arbitrary $v_0 \in \mathbf{C}^s$, the sharpest possible bound is $|v_n| \leq \|B^n\| \cdot |v_0|$. Clearly, the propagated error, and therefore the stability of the numerical approximation, depends on the uniform boundedness of $\|B^n\|$ for all $n$, that is, $M(B) <+\infty$.

In this paper, we analyze the power boundedness of composites of Toeplitz operators because such operators arise naturally in the numerical solution of intial value problems (c.f. Section 3 of \cite{Spi97}) and in differential equation-based models in science and engineering. Our results demonstrate the utility of reproducing kernel techniques in sourcing power-bounded operators that guarantee the numerical stability of approximation schemes. This work can help expand numerical methods to broader classes of differential equations and aid the development of novel numerical approximation schemes.  

\bmhead{Acknowledgments}

The author would like to thank Dr. Edward Timko for introducing him to this problem and mentoring him during his visit at the Georgia Institute of Technology School of Mathematics.

\end{document}